\documentclass[oneside,english]{amsart}
\usepackage[T1]{fontenc}
\usepackage[latin9]{inputenc}
\usepackage{amsbsy}
\usepackage{amstext}
\usepackage{amsthm}
\usepackage{amssymb,enumerate}
\usepackage[colorlinks=true]{hyperref}
\hypersetup{
	colorlinks=blue,%
	citecolor=red,%
	filecolor=black,%
	linkcolor=blue,%
	urlcolor=blue
}

\makeatletter
\theoremstyle{plain}
\newtheorem{thm}{\protect\theoremname}
\theoremstyle{plain}
\newtheorem{lem}[thm]{\protect\lemmaname}
\theoremstyle{plain}

 \def\ve{\varepsilon}
 
 \newcommand{\pa}{\partial}
  \newcommand{\R}{\mathbb{R}}
  \newcommand{\B}{\mathbb{B}}
   \newcommand{\ud}{\mathrm{d}}
\newcommand{\Sn}{\mathbb{S}^n}
\newcommand{\Sp}{\mathbb{S}}

\makeatother

\usepackage{babel}
\providecommand{\corollaryname}{Corollary}
\providecommand{\lemmaname}{Lemma}
\providecommand{\theoremname}{Theorem}

\allowdisplaybreaks

\begin{document}
\title[]{Conformal metrics of the ball with constant $\sigma_k$-curvature and constant boundary mean curvature}

\author{Xuezhang  Chen, Wei Wei}
\address{School of Mathematics \& IMS, Nanjing University, Nanjing 210093, P.R. China}
\email{xuezhangchen@nju.edu.cn,wei\_wei@nju.edu.cn}
\thanks{X. Chen is partially supported by NSFC (No.11771204). W. Wei is partially supported by NSFC No. 12201288, NSFJS No.BK20220755 and the Alexander von Humboldt Foundation.}

\begin{abstract}
On the upper hemisphere, we use the Obata-Escobar argument to classify
conformal metrics with constant $\sigma_k$ curvature and constant boundary mean
curvature in all types of cones including positive and negative
cones. This extends a result of Escobar in \cite{Es} for $k=1$.
\end{abstract}

\maketitle

\section{Introduction}

  In 1990,  Escobar initiated the study of an analogue of the Obata's geometric rigidity problem on closed  manifolds:  For a smooth compact manifold with boundary, under what conditions can an Einstein manifold be deformed conformally and still remain an Einstein manifold? An affirmative and complete  answer toward this problem for the ball can be found in \cite[Theorem 2.1]{Es}:
\emph{Let $g$ be a conformal metric on the unit Euclidean ball $(\mathbb{B}^{n+1},|\ud x|^{2})$ for $n\ge2$. If $g$ has constant scalar curvature $R_{g}$ and constant boundary mean
curvature $H_{g}$, then $g$ is Einstein.} However,  some cases in manifolds with boundary still remain open.

On a smooth Riemannian manifold $(M^{n+1},g)$, denote by
\[
A_{g}:=\frac{1}{n-1}\Big(\mathrm{Ric}_{g}-\frac{R_{g}}{2n}g\Big)
\]
the Schouten tensor, where $\mathrm{Ric}_{g}$ and $R_{g}$ are the Ricci tensor and the scalar
curvature in $g$, respectively. In the context of conformally invariant elliptic equations, $\sigma_{k}(A_{g})$
is one of the important objects and has played an important
role in conformal geometry, fully nonlinear equations and  mathematical general relativity. The $\sigma_{k}$ curvature  is always associated with a cone $\Gamma_{k}^{+}$ for ellipticity:
\[
\Gamma_{k}^{+}=\{\lambda\in \R^{n+1}; \sigma_{1}(\lambda)>0,\cdots, \sigma_{k}(\lambda)>0\}.
\]
Without confusion, we short $A_{g}\in\Gamma_{k}^{+}$ for $\lambda(A_{g})\in\Gamma_{k}^{+}$ while reserving $k$ for those in $\sigma_k$ and some quantities relative to $\sigma_k$.

Under the hypothesis that $|x|^{1-n}u(|x|^{-2}x)$ can be extended
to a positive $C^{2}$ function near $x=0$ for $2\le k\le n$, Viaclovsky
\cite{V1,V2} obtained the Liouville theorem for constant $\sigma_{k}(A_{g})$
in $\Gamma_{k}^{+}$. Concerning $k=2$ and $n=3$, Chang-Gursky-Yang
\cite{CGY1,CGY3} utilized the Obata argument for the entire space. For $n=3$,
the second named author with Fang and Ma \cite{FMW} introduced a
monotone formula with respect to level set of the solution, and then
proved a Liouville theorem for some more general $\sigma_{2}$-type curvature
equation including $\sigma_{2}$ Yamabe equation. In a series of works \cite{JLX,LiYY0,LL03,LL05,LLL,CLL1},
Y. Y. Li and his collaborators have formulated a complete theory for $\mathbb{R}^{n+1}$.
For $\mathbb{R}_{+}^{n+1}$, Y. Y. Li and his collaborators  \cite{LL6, CLL3} demonstrated
that the metric with positive constant $f(\lambda(A_{g}))$ in $\mathbb{R}_{+}^{n+1}$
with constant mean curvature on $\partial\mathbb{R}_{+}^{n+1}$ is exactly
the round metric up to a conformal diffeomorphism, where
$f(\lambda(A_{g}))$ is a fully nonlinear conformally invariant equation
including the $\sigma_{k}(A_{g})$ equation.  
Cavalcante-Espinar \cite{Cava-Espinar} ultilized an alternative approach to prove the Liouville
theorem in $\mathbb{S}_{+}^{n+1}$. Abanto-Espinar \cite{AE} established that
there is no conformal metric of the round metric on
$\mathbb{S}_{+}^{n+1}$ satisfying  the conformal invariant equation $f(\lambda(A_{g}))=0$
with negative constant boundary mean curvature. For zero $\sigma_{k}$
curvature, under certain assumption about the boundary mean curvature,
Case-Wang \cite{CW18} established the Obata type theorem in $\overline{\Gamma_{k}^{+}}$
for constant boundary $\mathcal{B}_{k}^{g}$ curvature from a natural
variational viewpoint - see S. Chen \cite{ChenSophie} for the definition of $\mathcal{B}_{k}^{g}$. The second named author \cite{W} used
the moving plane method to obtain the Liouville theorem for positive
constant $\sigma_{k}$ curvature in $\mathbb{S}_{+}^{n+1}$ with constant $\mathcal{B}_{k}^{g}$
curvature on $\partial \mathbb{S}_{+}^{n+1}$. Bao-Li-Li in \cite{CLL2} obtained the Liouville theorem in the half space $\mathbb{R}_{+}^{n+1}$ with constant $\mathcal{B}_{k}^{g}$
curvature on $\partial\mathbb{R}_{+}^{n+1}$.

In this paper, we use an Obata-Escobar argument for $\sigma_{k}$ and  boundary mean curvatures to build a complete picture as that of the
scalar and boundary mean curvatures in the ball and state our result as follows. 
\begin{thm}
\label{thm:obata for k-curvature}
Let $\tilde g$ be a conformal metric in the unit Euclidean ball $(\B^{n+1},|\ud x|^2)$ for $n \geq 2$ and $\tilde g$ has constant boundary mean curvature $\tilde H$ on $\partial \B^{n+1}$. For $2\le k\le n$, suppose one of three conditions is true:
\begin{enumerate}[(i)]
\item  $\tilde g$ has zero $\sigma_k$ curvature in $\overline{\Gamma_k^+}$. 
\item $\tilde g$ has positive constant $\sigma_k$ curvature in $\overline{\Gamma_k^+}$ .
\item $\sigma_k(-\tilde A)$ is a positive constant for $-\tilde A \in \overline{\Gamma_k^+}$. 
\end{enumerate}
 Then $\tilde g$ is Einstein. Moreover, there exist a  canonical conformal metric $g_{\mathrm{c}}$ and  a conformal diffeomorphism $\Phi$ on $\overline{\B^{n+1}}$ such that
$$\tilde g=\Phi^\ast (g_{\mathrm{c}})$$
and the metric $g_{\mathrm{c}}$ is explicitly given as follows:
\begin{enumerate}[(i)]
\item If $\tilde \sigma_k=0,\tilde H \in \R_+$, then  $g_{\mathrm{c}}=\tilde H^{-2} |\ud x |^2$.
\item If $\tilde \sigma_k>0, \tilde H \in \R$, then  $g_{\mathrm{c}}=\big(\frac{\binom{n+1}{k}}{2^k \tilde \sigma_k}\big)^{1/k}\left(\frac{2\ve}{|x|^2+\ve^2}\right)^{2}|\ud x|^2$,
where
$$\left(\frac{2^k \tilde \sigma_k}{\binom{n+1}{k}}\right)^{\frac{1}{2k}}\frac{\ve^2-1}{2\ve}=\tilde H \quad \Rightarrow \quad \ve=\left(\frac{\binom{n+1}{k}}{2^k \tilde \sigma_k}\right)^{\frac{1}{2k}}\left(\tilde H+\sqrt{\tilde H^2+\left(\frac{2^k \tilde \sigma_k}{\binom{n+1}{k}}\right)^{\frac{1}{k}}}\right).$$
\item If $\sigma_k(-\tilde A)>0, \tilde H>\big(\frac{2 \sigma_k(-\tilde A)}{\binom{n+1}{k}}\big)^{1/(2k)}$, then  $g_{\mathrm{c}}=\big(\frac{\binom{n+1}{k}}{2^k \sigma_k(-\tilde A)}\big)^{1/k}\left(\frac{2\ve}{\ve^2-|x|^2}\right)^{2}|\ud x|^2$, where
\begin{align*}
&\left(\frac{2^k \sigma_k(-\tilde A)}{\binom{n+1}{k}}\right)^{\frac{1}{2k}}\frac{\ve^2+1}{2\ve}=\tilde H \\
 \Rightarrow&~~ \ve=\left(\frac{\binom{n+1}{k}}{2^k \sigma_k(-\tilde A)}\right)^{\frac{1}{2k}}\left(\tilde H+\sqrt{\tilde H^2-\left(\frac{2^k \sigma_k(-\tilde A)}{\binom{n+1}{k}}\right)^{\frac{1}{k}}}\right)>1.
 \end{align*}
\end{enumerate}
Here, $\tilde A =A_{\tilde g}, \tilde \sigma_k=\sigma_k(\tilde A)$ and $\tilde H=H_{\tilde g}$.
\end{thm}

\section{Proof of main theorem}

For an $(n+1) \times (n+1)$ matrix $W$, we define  the $k^{\mathrm{th}}$ Newton transformation by
$$T_k(W)_j^i=\frac{1}{k!}\sum \delta\left(\begin{matrix}
			i_1&\cdots&i_k  &i   \\
			j_1 &\cdots& j_k&j
		\end{matrix}
		\right)
		W_{i_1}^{j_1}\cdots W_{i_k}^{j_k} ,
	$$
	where the Kronecker symbol 
	$\delta\left(\begin{matrix}
			i_1&\cdots&i_k    &i   \\
			j_1 &\cdots& j_k  &j
		\end{matrix}
		\right)
$
	has value  $1~ (\mathrm{resp.} -1)$ for an even (resp. odd) permutation in the index set $\{1,2,\cdots,n+1\}$ when $i_1,\cdots,i_k,i$ are distinct; otherwise it has value $0$. The following recursive formula holds (see Reilly \cite{Reilly}):
	$$T_k(W)_j^i=\sigma_k(W)\delta_j^i-T_{k-1}(W)_l^iW_j^l.$$

On a smooth manifold $(M^{n+1},g)$, we define
\[
L_{k}(A_{g})_j^i:=\frac{n+1-k}{n+1}\sigma_{k}(A_{g})\delta_j^i-T_{k}(A_{g})_j^i
\]
and 
\[
E_{g}=\mathrm{Ric}_{g}-\frac{R_{g}}{n+1}g.
\]
Then we have $\sigma_{1}(L_{k}(A_g))=0$ and from \cite[Lemma 6.8]{LL03} that
\begin{align*}
\langle L_{k}\left(A_{g}\right),E_{g}\rangle_g=(n-1)\left[ \frac{n+1-k}{n+1}\sigma_{k}\left(A_{g}\right)\sigma_{1}\left(A_{g}\right)-(k+1)\sigma_{k+1}\left(A_{g}\right)\right] .
\end{align*}
The tracefree $(1,1)$-tensor $L_{k}(A_{g})$ has proved to be very successful
when dealing with $\sigma_{k}$ Yamabe equations, see for example
\cite{CGY1,CGY3,G1,G2,G3,LL03,V1}. Moreover, on locally conformal flat manifolds $(M,g)$, it was proved in  \cite{V1} that $\nabla_{i} T_{k}(A_{g})_{j}^i=0$.

\begin{lem}[\protect{\cite[Lemma 6.8]{LL03}}]\label{lem:Newton inequality} For $1\leq k\leq n$
and $A_{g}\in\Gamma_{k}^{+}$, there holds
\[
\langle L_{k}\left(A_{g}\right),E_{g}\rangle_g \geq 0
\]
and equality holds if and only if $E_{g}=0$.
\end{lem}

We mimic the proof of \cite[Lemma 6.8]{LL03} to obtain the following result for the negative cone.
\begin{lem}
\label{lem:Newton for negative cone}
For $1\leq k\leq n$ and $-A_{g}\in\Gamma_{k}^{+}$, there holds
\[
\langle L_{k}\left(-A_{g}\right),-E_{g}\rangle_g\geq 0
\]
and equality holds if and only if $E_{g}=0.$
\end{lem}

In the following, let $g=u^{2}g_{0}$ in $\mathbb{S}_{+}^{n+1}:=\{x\in \Sp^{n+1}; x_{n+2}>0\}$ with $g_{0}=g_{\Sp^{n+1}}$ being the round metric, and $\bar g$ be the induced metric of $g$ on $\partial\mathbb{S}_{+}^{n+1}=\Sn$. Let $\nabla$ and $\overline \nabla$  denote the covariant derivatives with respect to the Levi-Civita connections of $g$ and $\bar g$. Then we obtain
\begin{equation}\label{eq:New  expression}
E_{g}=-\frac{n-1}{u}\big(\nabla^{2}u-\frac{\Delta_{g}u}{n+1}g\big)
\end{equation}
and
\begin{equation*}
A_{g}=-\frac{1}{u}\nabla_{g}^{2}u+\frac{\left|\nabla u\right|_{g}^{2}}{2u^{2}}g+\frac{1}{2u^{2}}g \qquad\mathrm{~~in~~} \quad \mathbb{S}_{+}^{n+1}.
\end{equation*}
Under Fermi coordinates around $\mathsf{p}\in \partial\mathbb{S}_{+}^{n+1}$  the metric $g$ can be expressed
		as 
		\[
		g=\ud y_{n+1}^{2}+\bar{g}_{\alpha\beta}\ud y_{\alpha}\ud y_{\beta}\qquad\qquad\mathrm{~~in~~}B_{\rho}^{+}:=\{y \in \R^{n+1}; |y|<\rho,~y_{n+1}>0\},
		\]
		where $(y_{1},\cdots,y_{n})$ are the geodesic normal coordinates
		on $\pa \mathbb{S}_{+}^{n+1}$ and $\nu_g=-\pa_{y_{n+1}}$ is the outward unit normal vector on $\pa \mathbb{S}_{+}^{n+1}$. We define the second fundamental form on $\pa \mathbb{S}_{+}^{n+1}$ by 
	\[
	L_{\alpha \beta}=\langle\nabla_{\pa _\alpha}\nu_g,\pa _\beta\rangle=-\langle\nabla_{\pa _\alpha} \pa _\beta,\nu_g\rangle, \quad 1 \leq \alpha,\beta \leq n.
	\]
Here, $\pa_\alpha=\pa _{y_\alpha}$,  $\nu_{g}$ is the outward unit
normal on $\partial\mathbb{S}_{+}^{n+1}$, and $H_{g}=(\sum_{\alpha=1}^n L_\alpha^{\alpha})/n$ is the mean curvature with respect to $g$. Moreover, we have
\begin{equation}\label{eqn:mean_curvature}
u_{\nu_{g}}=H_{g}u.
\end{equation}

\begin{lem}\label{lem:inequality for positive cone} 
Fix {$1\leq k\le n$}. For constant $\sigma_{k}(A_g)$
and $A_{g}\in\overline{\Gamma_{k}^{+}}$ in $(\mathbb{S}_{+}^{n+1},g)$,
we have 
\begin{align*}
0\le & \int_{\mathbb{S}^{n}}T_{k}(A)(\nabla u,\nu_{g})\ud\sigma_{g}-\frac{n+1-k}{n+1}\sigma_{k}(A_{g})\int_{\mathbb{S}^{n}} H_{g} u\ud \sigma_{g},
\end{align*}
and equality holds if and only if $g$
is Einstein.
\end{lem}
\begin{proof}
By \eqref{eq:New  expression} we have
\begin{align*}
0 \le& \int_{\mathbb{S}_{+}^{n+1}}\langle L_{k}\left(A_{g}\right),E_{g}\rangle_g u\ud v_{g}\\
 =&\int_{\mathbb{S}_{+}^{n+1}}\langle L_{k}\left(A_{g}\right),-\frac{n-1}{u}\nabla^{2}u\rangle_{g}u\ud v_{g}\\
=&\int_{\mathbb{S}_{+}^{n+1}}\langle \frac{n+1-k}{n+1}\sigma_{k}(A_{g})g-T_{k}(A_{g}),-(n-1)\nabla^{2}u\rangle_{g}\ud v_{g}\\
 =&(n-1)\int_{\mathbb{S}_{+}^{n+1}}\langle T_{k}(A_{g}),\nabla^{2}u\rangle \ud v_{g}-(n-1)\frac{n+1-k}{n+1}\sigma_{k}(A_{g})\int_{\Sp_{+}^{n+1}}\Delta_g u \ud v_{g}\\
 =&-(n-1)\int_{\mathbb{S}_{+}^{n+1}}\nabla_{i}T_{k}(A_{g})_{j}^i \nabla^j u \ud v_{g}+(n-1)\int_{\mathbb{S}^{n}}T_{k}(A_g)(\nabla u,\nu_{g})\ud\sigma_{g}\\
 &-(n-1)\frac{n+1-k}{n+1}\sigma_{k}(A_{g})\int_{\mathbb{S}^{n}}u_{\nu_{g}}\ud v_{g}\\
 =&(n-1)\int_{\mathbb{S}^{n}}T_{k}(A_{g})(\nabla u,\nu_{g})\ud\sigma_{g}-(n-1)\frac{n+1-k}{n+1}\sigma_{k}(A_{g})\int_{\mathbb{S}^{n}}H_{g} u\ud v_{g},
\end{align*}
where we used $\nabla_{i}T_{k}(A_{g})_j^i=0$ in $\mathbb{S}_{+}^{n+1}.$ 
\end{proof}

\begin{lem}
\label{lem:inequality for negative cone}Fix {$1\leq k\le n$}. For constant $\sigma_{k}(-A_{g})$
and $-A_{g}\in\overline{\Gamma_{k}^{+}}$ in $(\mathbb{S}_{+}^{n+1},g)$,
we have 
\begin{align*}
0\le & \int_{\mathbb{S}^{n}}-T_{k}(-A_{g})(\nabla u,\nu_{g})\ud\sigma_{g}+\frac{n+1-k}{n+1}\sigma_{k}(-A_{g})\int_{\mathbb{S}^{n}}H_{g} u\ud v_{g},
\end{align*}
and the equality holds if and only if $g$ is Einstein.
\end{lem}

\begin{proof}
By \eqref{eq:New  expression} and Lemma \ref{lem:Newton for negative cone}, we have
\begin{align*}
0 & \le\int_{\mathbb{S}_{+}^{n+1}}\langle L_{k}\left(-A_{g}\right),-E_{g}\rangle_g u\ud v_{g}\\
 & =\int_{\mathbb{S}_{+}^{n+1}}\langle L_{k}\left(-A_{g}\right),\frac{n-1}{u}\nabla^{2}u \rangle_{g}u\ud v_{g}\\
 & =\int_{\mathbb{S}_{+}^{n+1}}\langle \frac{n+1-k}{n+1}\sigma_{k}(-A_g)g-T_{k}(-A_g),(n-1)\nabla^{2}u \rangle_{g}\ud v_{g}\\
 & =-(n-1)\int_{\mathbb{S}_{+}^{n+1}}\langle T_{k}(-A_{g}),\nabla^{2}u \rangle \ud v_{g}+(n-1)\frac{n+1-k}{n+1}\sigma_{k}(-A_{g})\int_{\mathbb{S}_{+}^{n+1}}\Delta_{g}u\ud v_{g}\\
 & =-(n-1)\int_{\mathbb{S}^{n}}T_{k}(-A_{g})(\nabla u,\nu_{g})\ud\sigma_{g}+(n-1)\frac{n+1-k}{n+1}\sigma_{k}(-A_{g})\int_{\mathbb{S}^{n}}u_{\nu_{g}}\ud v_{g}\\
 & =-(n-1)\int_{\mathbb{S}^{n}}T_{k}(-A_{g})(\nabla u,\nu_{g})\ud\sigma_{g}+(n-1)\frac{n+1-k}{n+1}\sigma_{k}(-A_{g})\int_{\mathbb{S}^{n}}H_{g}u\ud v_{g},
\end{align*}
where we used $\nabla_{i}T_{k}(-A_g)_{j}^i=0$ in $\mathbb{S}_{+}^{n+1}.$ 
\end{proof}
In the following, let $A_{g}^{\top}$ denote the tangential part of
$A_{g}$ on the boundary $\mathbb{S}^{n}$.

\begin{lem}
\label{lem:Kadan-Warn -1}
Fix {$1\leq k\le n$}. Let $g=u^{2}g_{0}$
be a metric in $\mathbb{S}_{+}^{n+1}$ such that $\sigma_{k}(A_{g})$ is constant. Then

\[
\int_{\mathbb{S}^{n}}u\sigma_{k}(A_{g}^{\top}) \ud\sigma_{g}=\frac{n+1-k}{n+1}\sigma_{k}(A_{g})\int_{\mathbb{S}^{n}}u\ud\sigma_{g}.
\]
\end{lem}

\begin{proof}

 The following important identity was first proved by Zhen-Chao Han \cite[p.6 in \href{https://sites.math.rutgers.edu/~zchan/current-preprint/KW.pdf}{the preprint version}]{Han}: For  a conformal Killing vector field $X$ in $\mathbb{S}_{+}^{n+1}$, there holds
\[
X\big(\sigma_{k}(A_g)\big) =-\frac{n+1}{n+1-k}\nabla_{j}\left[T_{k}(A_{g})_{i}^{j}X^{i}\right]+\nabla_{j}\left[\sigma_{k}(A_{g})X^{j}\right].
\]
Alternatively, it is a direct consequence of  \cite[(4)]{Han} and the elementary property of $T_k(A_g)$ (see \cite{V1}) that 
$$T_k(A_g)_i^j=\sigma_k(A_g) \delta_i^j-T_{k-1}(A_g)_l^j A_i^l, \qquad T_{k-1}(A_g)_l^i A_i^l=k \sigma_k(A_g).$$

Integrating both sides of the above identity over $\Sp^{n+1}_+$ yields 
\begin{align*}
 & (n+1-k)\int_{\mathbb{S}_{+}^{n+1}}X\big(\sigma_{k}(A_g)\big) \ud v_{g}\\
= & -(n+1)\int_{\mathbb{S}_{+}^{n+1}}\nabla_{j}\left[T_{k}(A_g)_{i}^{j}X^{i}\right]\ud v_{g}+(n+1-k)\int_{\mathbb{S}_{+}^{n+1}}\nabla_{j}\left[\sigma_{k}(A_g)X^{j}\right]\ud v_{g}\\
= & -(n+1)\int_{\mathbb{S}^{n}}T_{k}(A_{g})(X,\nu_{g}) \ud\sigma_{g}+(n+1-k)\int_{\mathbb{S}^{n}}\sigma_{k}(A_{g})\langle X,\nu_{g}\rangle_{g} \ud\sigma_{g}.
\end{align*}

We now take $X=-\nabla_{g_0}x_{n+2}$ to be the conformal Killing vector field on $\mathbb{S}_{+}^{n+1}$ such that $X=u\nu_{g}$ on $\mathbb{S}^{n}$, and obtain that
\begin{align*}
0 & =-(n+1)\int_{\mathbb{S}^{n}}T_{k}(A_{g})(X,\nu_{g})\ud\sigma_{g}+(n+1-k)\int_{\mathbb{S}^{n}}\sigma_{k}(A_{g})\langle X,\nu_{g}\rangle_{g} \ud\sigma_{g}\\
 & =-(n+1)\int_{\mathbb{S}^{n}}u\sigma_{k}(A_{g}^{\top})\ud\sigma_{g}+(n+1-k)\int_{\mathbb{S}^{n}}u\sigma_{k}(A_{g})\ud\sigma_{g}.
\end{align*}
The proof is complete. 
\end{proof}
Now we prove a key identity involving the boundary term. 
\begin{lem}
\label{lem:boundary final term}
Fix  {$1 \leq k\le n$.} Let $g=u^{2}g_{0}$
be a metric in $\mathbb{S}_{+}^{n+1}$  such that $\sigma_{k}(A_{g})$  and $H_{g}$
are constant. Then 
\[
\int_{\mathbb{S}^{n}}T_{k}(A_{g})(\nabla u,\nu_{g})\ud\sigma_{g}=\frac{n+1-k}{n+1}\sigma_{k}(A_{g})H_{g}\int_{\mathbb{S}^{n}}u\ud \sigma_{g}.
\]
\end{lem}

\begin{proof}
Notice that $g$ is umbilical on $\pa \mathbb{S}_{+}^{n+1}$, that is, $L_{\alpha \beta}=H_g \bar g_{\alpha \beta}$. Under Fermi coordinates around $\mathsf{p} \in \pa \mathbb{S}_{+}^{n+1}$, by Codazzi equation we have
\begin{align*}
A_{\beta (n+1)}=\frac{1}{n-1}R_{\beta (n+1)}=&\frac{1}{n-1}R_{~\beta \alpha (n+1)}^\alpha\\
=&\frac{1}{n-1}({\overline{\nabla}}_\beta L^{~\alpha}_{\alpha}-{\overline{\nabla}}^\alpha L_{\alpha \beta})=\overline{\nabla}_\beta H_g.
\end{align*}
Then we obtain
\begin{align}\label{Tk}
T_{k}(A_{g})(\overline{\nabla}u,\nu_{g})=&-T_{k}(A_{g})_{n+1}^\alpha u_{,\alpha}\nonumber\\
=&T_{k-1}(A_{g}^\top)_\beta^\alpha A_{n+1}^\beta  u_{,\alpha}=T_{k-1}(A_{g}^{\top})(\overline{\nabla}H_g,\overline{\nabla}u),
\end{align}
where we used \cite[Lemma 3(d)]{ChenSophie} in the second identity.

 Hence, by \eqref{eqn:mean_curvature} and \eqref{Tk}, we obtain
\begin{align*}
 & \int_{\mathbb{S}^{n}}T_{k}(A_{g})(\nabla u,\nu_{g})\ud\sigma_{g}\\
= & \int_{\mathbb{S}^{n}}T_{k}(A_{g})(\overline{\nabla}u+u_{\nu_{g}}\nu_{g},\nu_{g})\ud\sigma_{g}\\
= & \int_{\mathbb{S}^{n}}T_{k}(A_{g})(\overline{\nabla}u,\nu_{g})\ud\sigma_{g}+\int_{\mathbb{S}^{n}}u_{\nu_{g}}T_{k}(A_g)(\nu_{g},\nu_{g})\ud\sigma_{g}\\
= & \int_{\mathbb{S}^{n}}u_{\nu_{g}}\sigma_{k}(A_{g}^{\top})\ud\sigma_{g}=H_{g} \int_{\mathbb{S}^{n}}u\sigma_{k}(A_{g}^{\top})\ud\sigma_{g},
\end{align*}
where we used $T_{k}(A)(\nu_{g},\nu_{g})=\sigma_{k}(A_{g}^{\top})$. Therefore, this together with Lemma \ref{lem:Kadan-Warn -1} yields the desired assertion.
\end{proof}
\begin{thm}\label{thm:nonnegative_cone}
Let $g$ be a conformal metric of $g_0$ with $A_{g}\in\overline{\Gamma_{k}^{+}}$ for  $2\leq k\le n$, and satisfy
\[
\begin{cases}
\sigma_{k}(A_{g})=\mathrm{constant}\ge0 &\qquad \mathrm{~~in~~} \mathbb{S}_{+}^{n+1},\\
H_{g}=\mathrm{constant} & \qquad \mathrm{~~on~~} \mathbb{S}^{n}.
\end{cases}
\]
Then $g$ is Einstein.
\end{thm}

\begin{proof}
By Lemma \ref{lem:inequality for positive cone} and Lemma \ref{lem:boundary final term} 
we have 
\begin{align*}
0\le & \int_{\mathbb{S}^{n}}T_{k}(A_g)(\nabla u,\nu_{g})\ud\sigma_{g}-\frac{n+1-k}{n+1}\sigma_{k}(A_{g})H_{g}\int_{\mathbb{S}^{n}}u\ud \sigma_{g}
= 0
\end{align*}
and then from Lemma \ref{lem:Newton inequality} that $g$
is Einstein. 
\end{proof}

\begin{thm}\label{thm:negative_cone}
Let $g$ be a conformal metric of $g_0$ with $- A_{g}\in\overline{\Gamma_{k}^{+}}$ for  $2\leq k\le n$, and satisfy

\[
\begin{cases}
\sigma_{k}(-A_{g})=\mathrm{constant} \ge0 &\qquad \mathrm{~~in~~} \mathbb{S}_{+}^{n+1},\\
H_{g}=\mathrm{constant} & \qquad \mathrm{~~on~~} \mathbb{S}^{n}.
\end{cases}
\]
Then $g$ is Einstein.
\end{thm}

\begin{proof}
Using Lemma \ref{lem:Newton for negative cone} and Lemma \ref{lem:boundary final term} with $-A_g$ in place of $A_g$,
we have
\begin{align*}
0\le & -\int_{\mathbb{S}^{n}}T_{k}(-A_g)(\nabla u,\nu_{g})\ud\sigma_{g}+\frac{n+1-k}{n+1}\sigma_{k}(-A_{g})H_{g}\int_{\mathbb{S}^{n}}u\ud \sigma_{g}\\
= & (-1)^{k+1}\frac{n+1-k}{n+1}\sigma_{k}(A_{g})H_{g}\int_{\mathbb{S}^{n}}u\sigma_{g}+\frac{n+1-k}{n+1}\sigma_{k}(-A_{g})H_{g}\int_{\mathbb{S}^{n}}u\ud \sigma_{g}\\
= & 0.
\end{align*}
Then it follows from Lemma \ref{lem:Newton for negative cone} that $g$
is Einstein. 
\end{proof}

\begin{proof}
[\textbf{Proof of Theorem \ref{thm:obata for k-curvature}}] Using a stereographic projection $\pi:\Sp^{n+1}\backslash\{S\} \to \mathbb{R}^{n+1}$ from the south pole $S$, we identify $(\B^{n+1},|\ud x|^2)$ and the hemisphere $\Sp^{n+1}_+$ equipped with the round metric $g_{\Sp^{n+1}}$.

We abuse the notation $g$ for a while and set $g=|\ud x|^2$. By Theorems \ref{thm:nonnegative_cone} and \ref{thm:negative_cone}, we know that  the metric $\tilde g:=\phi^{-2}|\ud x|^2$ is Einstein, where $\phi$ is a positive smooth  function in $\overline{\B^{n+1}}$.  Then we have
\begin{align*}
0=\tilde E=(n-1)\phi^{-1}(\nabla^2 \phi-\frac{\Delta \phi}{n+1}g),
\end{align*}
that is,
\begin{equation}\label{trace free}
\phi_{ij}=\frac{\Delta \phi}{n+1} \delta_{ij}.
\end{equation}
Taking derivatives yields $\nabla_i\Delta \phi=\phi_{ijj}=\frac{1}{n+1}\nabla_i\Delta \phi$. This gives $ \nabla \Delta \phi=0$ and then $\Delta \phi=\mathrm{const.}$  
Hence, up to a rigid motion of $\R^{n+1}$ we obtain $\phi(x)=a|x-y|^2+c$ for  $y\in \mathbb{R}^{n+1}$ and  $a,c \in \mathbb{R}$.

A direct computation yields
\begin{align*}
\tilde A_{ij}=(\log \phi)_{ij}+(\log \phi)_i(\log \phi)_j-\frac{1}{2}|\nabla \log \phi|^2 g_{ij}=2ac \tilde g \qquad \mathrm{~~in~~} \B^{n+1}.
\end{align*} 
Thus we obtain
$$\tilde \sigma_k=\binom{n+1}{k} (2ac)^k$$
and on $\partial \B^{n+1}$, 
$$\tilde H=a|x-y|^2+c-2a(x-y)\cdot x=a(|y|^2-1)+c.$$
Notice that there are three unknown variables (i.e. $a,c$ and $|y|$) in the above two equations.

We divide our discussion into three cases:
\begin{enumerate}[(i)]
\item If $\tilde \sigma_k=0$, then we have either $c=0,a=1, |y|>1$ with $\tilde H=|y|^2-1$, due to the boundness of $g$,  or  $a=0,\tilde H=c\equiv \phi>0$. We aim to show that two metrics $g_1=|x-y|^{-4}|\ud x|^2$ and $g_2=\tilde H^{-2} |\ud x|^2$ are conformally equivalent. Indeed,  we define a conformal transformation on $\overline{\B^{n+1}}$ by
\begin{equation}\label{transformation:conf}
x=\varphi_y(z)=y+\frac{(|y|^2-1)(z-y)}{|z-y|^2},
\end{equation}
 which is a conformal transformation on $\overline{\B^{n+1}}$ with the property that
$$\varphi_y^\ast(|\ud x|^2)=\frac{(|y|^2-1)^2}{|z-y|^4} |\ud z|^2=(|y|^2-1)^2 g_1=\tilde H^2 g_1.$$
This implies that  $\varphi_y^\ast(g_2)=g_1$, so we just take $g_{\mathrm{c}}=g_2$.

\item If $\tilde \sigma_k>0, \tilde H\in \mathbb{R}$, then $a,c>0, y \in \mathbb{R}^{n+1}$  by the positivity of $\phi$. So we may take $y=0$ and $g_{\mathrm{c}}, \ve$ as desired.
\item If $\sigma_k(-\tilde A)=(-1)^k \binom{n+1}{k} (2ac)^k>0$   and $\sigma_1(-\tilde A)=-2(n+1) ac>0$, then $ac<0$.

(iii$_a$) When $|y| \leq 1$. If $a>0$, then $c<0$, however $\phi(y)<0$, which is impossible. If $a<0$, then $c>0$. Since $\phi>0$ in $\overline{\B^{n+1}}$, we obtain
\begin{align*}
\min_{\overline{\B^{n+1}}}\phi(x)=a(1+|y|)^2+c>0\quad \Longrightarrow \quad |y|<\sqrt{-\frac{c}{a}}-1.
\end{align*}
Then it gives
\begin{align*}
\tilde H=a(|y|^2-1)+c>2\sqrt{-ac}=\left(\frac{2 \sigma_k(-\tilde A)}{\binom{n+1}{k}}\right)^{\frac{1}{2k}}.
\end{align*}

(iii$_b$) When $|y|>1$, we distinguish it into two cases. 

($b_1$) If $a<0$, then $c>0$, and the  same argument above goes through. 

($b_2$) If $a>0$, then $c<0$. Since $\phi>0$ in $\overline{\B^{n+1}}$, we obtain
\begin{align*}
\min_{\overline{\B^{n+1}}}\phi(x)=a(|y|-1)^2+c>0 \quad \Longrightarrow \quad |y|>\sqrt{-\frac{c}{a}}+1.
\end{align*}
This implies
$$\tilde H=a(|y|^2-1)+c>2\sqrt{-ac}=\left(\frac{2 \sigma_k(-\tilde A)}{\binom{n+1}{k}}\right)^{\frac{1}{2k}}.$$
 With the same conformal transformation $\varphi_y$ as in \eqref{transformation:conf}, a direct computation yields
\begin{align*}
&\varphi_y^\ast\left((a|x-y|^2+c)^{-2}|\ud x|^2\right)\\
=&\left(a\frac{(|y|^2-1)^2}{|z-y|^2}+c\right)^{-2}\frac{(|y|^2-1)^2}{|z-y|^4} |\ud z|^2\\
=&\left(a(|y|^2-1)+\frac{c}{|y|^2-1} |z-y|^2\right)^{-2} |\ud z|^2.
\end{align*}
If we let $c'=a(|y|^2-1)>0,a'=\frac{c}{|y|^2-1}<0$, then $a'c'=ac$ and $a'(|y|^2-1)+c'=c+a(|y|^2-1)$. In this sense, these two metrics obtained in cases ($b_1$) and ($b_2$) are equivalent up to a conformal transformation on $\overline{\B^{n+1}}$.

In summary, the canonical conformal metric reduces to the case that  $a<0$ and $c>0$, $\tilde g=(a|x-y|^2+c)^{-2}|\ud x|^2$ for all $y \in \mathbb{R}^{n+1}$. Due to this reason,  we may take $y=0$ and $g_{\mathrm{c}}, \ve$ as in (iii).
\end{enumerate}
\end{proof}

\end{document}